\documentclass[a4paper]{amsart}
\usepackage{amsmath,amsthm}
\usepackage{amssymb, enumitem, mathrsfs, verbatim,bm}
\usepackage{color} 
\usepackage[latin1]{inputenc}
\numberwithin{equation}{section}
\DeclareMathOperator*{\esssup}{ess\,sup}
\def\O{\Omega}

\def\n{\nabla}
\def\pd{\partial}
\def\iO{\int_\O}
\def\intl#1{\int\limits_{#1}}
\def\intll#1#2{\int\limits_{#1}^{#2}}
\def\dm{|\hskip-0.05cm|}
\def\displ{\displaystyle}

\def\VS{\vspace{6pt}\\\displ }
\def\rf#1{{\rm(\ref{#1})}}
\def\chiu{\hfill$\displaystyle\vspace{4pt}
\underset\Box\null$\par}

\def\R{\Bbb R}
\def\N{\Bbb N}

\def\be{\begin{equation}}
\def\ba{\begin{array}}
\def\ea{\end{array}}
\def\ee{\end{equation}}

\def\div{\mathop{\mathrm{div}}\nolimits}
\def\per{\mathop{\mathrm{per}}\nolimits}

\def\vec#1{\boldsymbol{#1}}
\def\bv{\vec{v}}
\def\bx{\vec{x}}
\def\bu{\vec{u}}
\def\bw{\vec{w}}
\def\bF{\vec{f}}
\def\bB{\vec{B}}
\def\bA{\vec{A}}
\def\S{\vec{S}}
\def\b{\vec{b}}
\def\bphi{\vec{\Phi}}
\def\bph{\vec{\varphi}}

\def\e{\vec{e}}
\def\x{x}
\def\cD{\mathcal{D}}
\def\bo{\vec{\omega}}
\title{Existence of regular time-periodic solution to shear-thinning fluids.}\date{}
\author{Anna Abbatiello and Paolo Maremonti}\thanks{Dipartimento di Matematica e Fisica, Universit\`a degli Studi della Campania ``Luigi Vanvitelli", via A. Vivaldi  n.43, 81100 Caserta, Italy. Email addresses: {\tt anna.abbatiello@unicampania.it, paolo.maremonti@unicampania.it} }

\newtheorem{thm}{Theorem}[section]
\newtheorem{cor}[thm]{Corollary}
\newtheorem{lem}[thm]{Lemma}

\theoremstyle{definition}
\newtheorem{defin}[thm]{Definition}
\newtheorem{rem}[thm]{Remark}

\newtheorem{ineq}[thm]{Inequality}
\numberwithin{equation}{section}

\begin{document}
\maketitle

\begin{abstract}
In this note we investigate the existence of time-periodic solutions to the $p$-Navier-Stokes system in the singular case of $p\in(1,2)$,  that describes the flows of an incompressible shear-thinning fluid.  
In the $3D$  space-periodic setting and for $p\in[\frac53,2)$ we prove the existence of a regular time-periodic solution corresponding to a  time periodic  force data which is assumed small in a suitable sense. As a particular case we obtain ``regular'' steady solutions.
\end{abstract}
\section{Introduction}
The  existence of time periodic solutions for shear-thinning and electro-rheological fluids was firstly studied in \cite{Lions-per}, where a \underline{modified} $p$-Navier-Stokes case for $p\geq 3$ constant is considered, and in \cite{C} for the $p(t,\x)$-Stokes case. The main purpose of this note is to investigate the question for 
the full $p$-Navier Stokes system. However the difficulties of the problem lead to consider a special case of the problem if compared with the ones 
studied in  \cite{C, Lions-per}. Nevertheless, our results are on the wake of the ones obtained  in \cite{BL}\,\footnote{\, The authors are grate to
 the editor in chief Prof. G.P. Galdi who submitted  to their attention the interesting paper \cite{BL}.}.  To better introduce the   results contained in  \cite{BL} and to show the ours  related to shear-thinning fluids, we state the special problem studied in this note.  
 We consider  \be\label{system}
\ba{l}\displ
\bv_t + \div(\bv\otimes \bv) - \div \S(\cD\bv) +\nabla \pi =  \bF \ \mbox{in} \ (0, T)\times\O, \\ \displ
\div\bv = 0 \ \mbox{in} \ (0, T)\times\O, \\ 
\ea\ee
where $\bv=(v_1, v_2, v_3) $ is the vector field of velocity, $\bv_t$  is the time derivative of $\bv$, $\S$ is the deviatoric part of the Cauchy stress tensor constitutively determined by
\be\label{S}
\S(\cD\bv)=|\cD\bv|^{p-2}\cD\bv \ \ \mbox{with} \ p\in (1, 2),
\ee
where  $\cD\bv:=\frac{1}{2}(\n\bv+(\n\bv)^T)$ symmetric part of the spatial gradient of $\bv$, $\pi$ is the pressure scalar field or mean normal stress,  finally  $\bF=(f_1, f_2, f_3) $ is the prescribed body force.   We study the system \eqref{system} with \eqref{S} in $(0, T)\times\O$ where $T$ is a positive real number and $\O = (0, 1)^3 \subset \R^3$. We endow the problem with space-periodic boundary
 conditions, i.e. we require 
\be\ba{l}\displ\label{periodicity}
\bv_{|\Gamma_j}=\bv_{|\Gamma_{j+3}}, \  \n\bv_{|\Gamma_j}=\n\bv_{|\Gamma_{j+3}},\VS
\pi_{|\Gamma_j}=\pi_{|\Gamma_{j+3}},
\ea
\ee
where $\Gamma_j:=\pd\O\cap\{x_j=0\}$ and $\Gamma_{j+3}:=\pd\O\cap\{x_j=1\}$ with $j=1,2,3$.
  \par In \cite{BL}, assuming  $p\in (\frac95,2)$ and 
  the data force $\bF$  time periodic of period $T>0$,   the authors realize the existence of weak solutions which
   are  time periodic with period $T$. They achieve the result of existence by employing  a fixed point  theorem. For analogous questions related to  the Newtonian fluids, this approach is essentially due to Prodi by the Tychonov fixed point theorem   and to Prouse by the Brouwer fixed point theorem,  see respectively \cite{PR} and  \cite{Pru}.  Brouwer's fixed point theorem  working in the finite dimensional case, it is employed on the Galerkin approximations (that are finite combinations of elements). Then there is the question of the convergence to a limit of the Galerkin approximation. In \cite{BL} they solve the convergence questions requiring a restriction to the power  $p\in(\frac95,2)$.    \par In this note   we follow a   different point of view. In the sense that we look for solutions corresponding to small data. This allows us to exhibit  the results for a  wider range of power-law index $p$ and ``more regular solutions''.  More precisely,
assuming in (\ref{system}) $p\in \left[\frac53,2\right)$,   we are able to obtain the existence of a time-periodic solution to \eqref{system}, subject to a small time-periodic force. The proof employs an idea by Serrin exhibited in \cite{S}. 
Assuming the existence of    solutions defined for all $t>0$ and their regularity,    in the case of $3D$-Navier-Stokes  boundary value problem subject to a time periodic force, in \cite{S} Serrin proves that any solution asymptotically tends to a time periodic solution whose time-period is the same of the force.
  This approach is formal in the sense that the assumptions on the existence of the solutions enjoying the regularity required by Serrin is still an open problem. Nevertheless, as made in \cite{M}  for $\Omega\equiv\mathbb R^3$,   for small data and in suitable function spaces Serrin's approach can be used with success.     
\par In this note we follow the same approach of \cite{M}. Firstly, for $p\in  [\frac{5}{3}, 2)$,  under the assumption of small data, we furnish the existence of a regular solution (Theorem\,\ref{thmA}), in a   sense   specified later, to the system \eqref{system} with \eqref{S} and \eqref{periodicity}. Then, under the assumption of periodicity for force data  we realize the time periodic solution. It is the case to make precise the claim ``small data force''. Actually we assume that   $\bF\in L^\infty(0,\infty;L^q(\O))$, that, as in the case of the Navier-Stokes equations,
 makes the difference since no time-integrability properties are required. The smallness is required for $\esssup_{t\geq0} \dm \bF(t)\dm_q$. As a particular case we consider $f$ independent of $t$, which allows us to deduce the existence of steady solutions. \par Our results of existence seem of some interest in connection with the fact they are obtained   for the singular case    $p\in[\frac53,2)$. We point out that the limit case of $p=\frac53$ and the regularity $\nabla^2\bv\in L^2(0,T;L^\frac4{4-p}(\Omega))$ is achieved by a trick, employed as in \cite{CGM},   connected with the reverse  H\"older's inequality (see Inequality\,\ref{ineq} and Lemma\,\ref{lemma6}).
This represents an improvement with respect the results already known. Actually, in \cite{MA95} for the problem of kind \eqref{system}-\eqref{periodicity} provided $\bF\in L^{p'}((0,T)\times\Omega)$ suitable small together with the initial data, it is  showed the existence of a solution with $\nabla^2\mathbf{v}\in L^2((0, T); W^{2,p}(\Omega)) $ for $p\in(\frac53,2)$.
 In \cite{BDR} for the evolution problem they exhibit a solution local in time subject to an integrable force $\bF\in L^\infty(0,T; W^{1,2}(\Omega))\cap W^{1,2}((0,T); L^2(\Omega))$ with  $\nabla^2\bv\in L^{\frac{p(5p-6)}{2-p}}((0, T); W^{2,\frac{3p}{p+1}}(\Omega)) $ for $p\in(\frac75,2)$, and for the steady problem $\nabla^2\bv\in W^{2,\frac{3p}{p+1}}(\Omega) $  where it is assumed $p\in (\frac95,2)$ and space-periodic domains as in our case.\par We   restrict our considerations to the tridimensional case just for the simplicity, but the technique also works in the two-dimensional case and   the results hold for $p\in(1,2)$. \par In order to better explain  our results we introduce some spaces of functions. For the sequel it is worth to note that vector valued functions are printed in boldface while scalar ones in italic mode and we do not distinguish between space of scalar functions and space of vector-valued functions.  \par
As stated before let $\O$ be the  cube $(0, 1)^3 $  in $\R^3$ of points $\x=(x_1, x_2,x_3)$ and let $\O_T$ denote $(0, T)\times\O$.  For a domain $G$, that can be $\O$ or $\O_T$, for $q\geq 1$ and  $m\in \N$ we keep the notation $(L^q(G), \|\cdot\|_{q,G})$ and $(W^{m,q}(G), \|\cdot\|_{m,q,G})$ for the standard Lebesgue and Sobolev spaces and their associated norms.  
The subscript ``{$_{\rm per}$}" means that space-periodic functions having zero mean value are considered (i.e. each considered function $f$ satisfies $f(\x+\e_i)=f(\x)$, $i=1, 2, 3$ where $(\e_1, \e_2,\e_3)$ is the canonical basis of $\R^3$ and $\int_\O f=0$). Next let us establish the notation for the spaces of solenoidal space-periodic functions with values in $\R^3$  by setting
$$\mathscr{C}_{\rm per}(\O):=\left\{ \bphi\in \mathcal{C}^\infty_{\rm per}(\O, \R^3), \int_\O \bphi(\x)\,d\x=0, \div\bphi=0 \mbox{ in } \O\right\},$$
and 
$$J^q_{\rm per}(\O):=\overline{\mathscr{C}_{\rm per}(\O)}^{\|\cdot\|_q}, \ \ J^{m,q}_{\rm per}(\O):=\overline{\mathscr{C}_{\rm per}(\O)}^{\|\cdot\|_{m,q}}.$$ Since there is no confusion, in the sequel we omit the subscript $\null_{\rm per}$.
In particular when $q=2$ we denote $J(\O):=J^2(\O)$. 
Finally, if $X$ is a Banach space then $L^q(0, T; X)$ stands for the Bochner space of functions from $(0, T)$ to $X$ and such that their norm in $X$ is $q$-integrable, as well $ {C}(0,T; X)$ is the space of continuous functions from $[0, T]$ to $X$ endowed of the standard norm of the ${\rm sup}$. 
\par The subscript ``$t$'' denotes the differentiation with respect to time, the symbol $\cD\bph$ is the symmetric part of the gradient of a sufficiently smoooth vector function $\bph$, i.e.  $\cD\bph := \frac{1}{2}(\n\bph+(\n\bph)^T)$,  $(\cdot, \cdot)$ is used for the scalar product in $L^2(\O)$. 
\begin{defin}
A field $\bu: (0, +\infty) \times \O \to\R^3$ is said to be {\emph{time-periodic}} with period $\mathcal{T}$ if 
\begin{itemize}
\item[(i)] $\displaystyle\sup_{t\in[0,\mathcal{T}]} \|\bu(t)\|_X<+ \infty$ with $(X, \|\cdot\|_X)$ a Banach space,
\item[(ii)] $\|\bu(t+\mathcal{T})-\bu(t)\|_X=0 \ \forall t \geq 0.$
\end{itemize}
\end{defin}

Let us give the notion of regular solution to system \eqref{system} with $\S$ defined by \eqref{S}, endowed of boundary conditions \eqref{periodicity} and completed by providing an initial data $\bv_0$.
\begin{defin}\label{def-d2} For all $T>0$,
let $\bF\in L^{\infty}(0,T; L^{\frac{4}{p}}_{\per}(\O))$ and $\bv_0\in J_{\per}^{1,2}(\O)$, a  field $\bv:(0,T)\times \O \to\R^3$ is said to be {solution} to system \eqref{system} with $\eqref{S}$, \eqref{periodicity} and initial data $\bv_0$ if 
\begin{itemize}
 \item[(i)] $\nabla^2\bv\in L^2(0,T; L_{\per}^{\frac{4}{4-p}}(\O))$, $\bv\in C(0, T; J_{\per}^{1,2}(\O)),$ $ \bv_t\in 
 L^2(\O_T),$
\item[(ii)] $( \bv_t, \bo) + (\bv\cdot\nabla \bv, \bo) +  (\S(\cD\bv), \cD\bo) =  (\bF, \bo)$ for all $\bo\in J_{\per}^{1,p}(\Omega)$ and a.e. in $(0, T)$,
\item[(iii)] $\lim_{t\to0^+} \|\bv(t)-\bv_0\|_2=0$.
\end{itemize}
\end{defin}
The main theorems are stated below and they are proved in Section \ref{sectionA} and \ref{section1} respectively. \par
We denote by $\Lambda$ and $K$ two fixed constants such that
 \begin{equation}\label{constant-Lambda}\begin{split}
& \Lambda<\min\mbox{$\left\{  \left(\frac{p}{8C_S}\right)^{\frac{2}{3-p}}, \left( \frac{3^{\frac{2-p}{2}}C_SC_K}{2}\right)^{-\frac{2}{3-p}}, (C_S^{-1}C_K^{-1}3^{\frac{p-2}2}(p-1))^{\frac{2}{3-p}}\right\}$}\\
& \mbox{\; and \;}
 C_SK^2 2^{2-p}\leq \frac{p(p-1)}{8 }\Lambda^{p-1}\,, 
 \end{split}
 \end{equation} 
 where  the constants $  C_S,  $ depends on $\Omega$  and Sobolev inequalities, while $C_K$ is the constant of the Korn inequality.
\begin{thm}\label{thmA}
Let $\O$ be a three dimensional cube, $p\in \left[\frac{5}{3},2\right)$.  
If  $\bF\in L^\infty(0,+\infty;L^{\frac{4}{p}}(\O))
$  and $\bv_0\in J^{1, 2}(\O)$ with 
\begin{equation}\label{smallv0}
 \|\nabla\bv_0\|_{L^{2}(\O)}^2<\Lambda 
 \mbox{\; and \;}
\esssup_{t\geq0}  
 \|\bF(t)\|_{L^{\frac{4}{p}}(\O)}\leq K \,, 
\end{equation} 
then there exists a solution $\bv(t,\bx)$ to system \eqref{system} in the sense of Definition \ref{def-d2} such that 
\begin{equation}\label{gradient2}\sup_{t\in[0,T]} \|\nabla\bv(t)\|_{2}^2\leq \Lambda \ \ \forall T>0.\end{equation}
\end{thm}
\begin{rem}
Our result slightly improves  the one for $p\in(\frac{5}{3},2)$ and for the operator with some $\mu>0$ in the negative power of the degenerate operator $\S$ in\cite{MA95}(see also  \cite[Theorem 4.51]{MNR}). Moreover, in connection with problem \rf{system}, Theorem\,\ref{thmA} is the first which proves global existence assuming small force data in $L^\infty(0,\infty;X)$, with $X$ Banach space. For the proof we modify the techniques from \cite{BDR} and \cite{M}, where the authors consider   respectively   $p$-Navier-Stokes and Navier-Stokes systems.
The regularity of the $p$-Navier-Stokes system is still open and there are many papers dealing with it.  It has to be pointed out that the difficulties appear in the $p$-parabolic case already, and recently the $L^\infty(\varepsilon, T; W^{2,q}(\O))$ regularity up to the boundary  with $\varepsilon>0$ and $q\geq2$ is given for both domains, bounded or exterior, see \cite{CGM} and \cite{CM}.
Other relevant results concerning the regularity of the system of type \eqref{system} in the steady and evolutionary cases are given in  \cite{Beir},   \cite{BdV}, \cite{CG}, \cite{CG9}, \cite{DR}, \cite{MNR}.
\end{rem}
\begin{thm}\label{thm1}
Let $\O$ be a three dimensional cube, $p\in \left[\frac{5}{3},2\right)$ and let $\bF$ be a time-periodic function with period $\mathcal T$ fulfilling the assumptions in Theorem \ref{thmA}
Then corresponding to $f$ there exists a unique $\bv_0(\bx) \in J^{1, 2}_{\per}(\O)$  such that
$$ \|\nabla\bv_0\|_{L^{2}(\O)}^2<\Lambda$$
and solution  $\bv(t,\x)$    to system \eqref{system}, in the sense of Definition\,\ref{def-d2}, which is  time-periodic with period $\mathcal T$,  and that fulfills
\begin{equation}\label{gradient3}\sup_{t\in[0,\mathcal T]} \|\nabla\bv(t)\|^2_{2}\leq \Lambda. \end{equation}
\end{thm} 
It well known the extinction of a $p$-parabolic solutions for $p\in(1,2)$ (e.g. \cite{DiB}, \cite{CGM}). The same holds for $p$-Stokes problem (see e.g. \cite{C}, \cite{Anna}) and the following theorem is a special case of the one proved in \cite{C}:
\begin{thm}\label{extinction} Assume that data $\bF$ in system \rf{system}  verifies the 
  assumptions of Theorem\,\ref{thm1}. Let $\bv(t,x)$ be the time periodic solution. If the data $f$ admits  an extinction   instant $t_f\in(0, \mathcal T)$, that is 
$$\|\bF(t)\|_{\frac{4}{p}}=0 \mbox{ a.e. in } [t_f,\mathcal  T],$$ and $t_f+\frac{C_S^{3-p}\Lambda^{2-p}}{2-p}<\mathcal T$, then the time periodic solution $\bv$ admits an instant $\bar t$ such that
 $$\|\bv(t)\|_{L^{2}(\O)}=0  \mbox{ in } [\bar{t}, \mathcal T] \mbox{ with } \bar{t}\leq t_f+\frac{C_S^{3-p}\Lambda^{2-p}}{2-p},$$
where the constant $\Lambda$ is given  in \rf{constant-Lambda}.
\end{thm}
\begin{rem} Since our study is only qualitative,
it is worth to highlight that Theorem \ref{extinction} gives an estimate of the instant $\bar{t}$ and a lower bound for $\mathcal T$  that could not be optimal.
\end{rem}
\begin{cor}\label{cor} For $p\in \left[\frac{5}{3}, 2\right)$, assume that in system \rf{system}
$\bF$ be independent of $t$ and satisfying  the bound given in Theorem\,\ref{thmA}. Then there exists a steady solution, in the sense of Definition \ref{def-d2},  $\bv(\bx)$ to system \eqref{system}.\end{cor}

\section{Approximating system and some auxiliary results}
Let us introduce the following approximating system:
\be\label{system-mu}
\ba{l}\displ
\bv_t + \div(\bv\otimes \bv) - \div \S_\mu(\cD\bv) +\nabla \pi =  \bF \ \mbox{in} \ (0, T)\times\O, \VS
\div\bv = 0 \ \mbox{in} \ (0, T)\times\O, \\ 
\ea\ee
where $$\S_\mu(\cD\bv):=(\mu+|\cD\bv|^2)^{\frac{p-2}{2}}\cD\bv \  \ \ \mbox{with} \ \mu\in(0,1) \mbox{ and } p\in(1, 2),$$
endowed of space-periodic boundary conditions.
The operator $\S_\mu$ enjoys the following property,  that is proved e.g.  in \cite[Lemma 6.3]{DER}, 
\be\label{property-diff}\vspace{10pt}
(\S_\mu(\bA)-\S_\mu(\bB), \bA-\bB)\simeq |\bA-\bB|^2( \mu+ |\bA|+|\bB|)^{p-2}\geq 3^{\frac{p-2}2}|\bA-\bB|^2(\mu+|\bA|^2+|\bB|^2)^\frac {p-2}2\,,
\ee
for  $\bA,\bB\in \R^{3\times 3}_{\rm sym}$, where the symbol $\R^{3\times 3 }_{\rm sym}$ stands for the space of symmetric matrix of dimension $3$. 
\\
In what follows we prove some preliminar  results for the proof of the main theorems. Since in this paper we are interested in the singular case of the operator $\S$ (being $p\in(1,2)$, cf. \eqref{S}), our strategy is to provide the existence of regular solutions to system \eqref{system-mu}, and  then consider the passage to the limit as $\mu\to 0$. 
  \par
Let us construct a solution to system \eqref{system-mu} by employing the Faedo-Galerkin method. For the reader's convenience, we recall some basic relations (cfr. \cite{MNR}). Consider the eigenfunctions $\{\bo_k\}_{k\in\N}$ of the Stokes operator $\mathcal{A}$ and let $\{\lambda_k\}_{k\in \N}$ be the corresponding eigenvalues. Define $X_n=[\bo_1,\dots,\bo_n]$ and let $P^n$ be the projection $P^n: (W^{1,2}(\O), \|\cdot\|_{1,2}) \to (X_n, \|\cdot\|_{1,2})$, defined as $P^n \bu= \sum_{k=1}^n (\bu, \bo_k) \bo_k $, which satisfies $\|P^n\bu\|_{1,2}\leq \|\bu\|_{1,2}$ (see \cite{MNR}). We seek the Galerkin approximation 
\begin{equation}\label{def-galerkin}
\bv^n(t, \bx):= \sum_{k=1}^n c^n_k(t)\bo_k(\bx),
\end{equation}
 as solution to the following system of ordinary differential equations
\be\label{galerkin}
(\bv_t^n, \bo_k) + (\bv^n\cdot\nabla \bv^n, \bo_k) +  (\S_\mu(\cD\bv^n), \cD\bo_k) =  (\bF, \bo_k) \ \ k=1,\dots, n
\ee
with initial condition $\bv^n(0)= P^n\bv_0$ where $\bv_0$ is the prescribed initial data.  From Carath\'{e}odory Theorem follows the local in time existence of $\bv^n(t, \bx)$ in a time interval $[0, t_n)$,  $t_n\leq T$. In Lemma \ref{lemma8} we prove  that $t_n=T$. To this end we provide the following relations for the Galerkin sequence. Multiplying \eqref{galerkin} respectively by $c^n_k$, $\lambda_k c^n_k(t)$ and $\dot{c}^n_k(t)$,  summing over $k=1,\dots,n$, we get 
\begin{subequations}\label{galerkin1}\begin{align}
\frac{1}{2}\frac{d}{dt}\|\bv^n(t)\|_{2,\O}^2+(\S_\mu(\cD\bv^n), \cD\bv^n)& = (\bF(t), \bv^n(t)), \label{energy-d2}\\
\frac{1}{2}\frac{d}{dt}\|\nabla\bv^n(t)\|_{2,\O}^2+ \!\!\iO\! \n\S_\mu(\cD\bv^n)\cdot \n\cD\bv^n dx&=(\bv^n\!\cdot\!\n\bv^n\!, \Delta\bv^n) - (\bF\!,\Delta\bv^n), \label{grad1}\\
\|\bv^n_t(t)\|_2^2+ (\S_\mu(\cD\bv^n), \cD\bv^n_t)&=  (\bv^n\!\cdot\!\n\bv^n\!, \bv^n_t) + (\bF\!, \bv^n_t), \label{vt}
\end{align}\end{subequations}
where in \eqref{energy-d2}  we employed that the term $(\bv^n\!\cdot\!\nabla \bv^n\!, \bv^n)$ vanishes since $\bv^n$ is solenoidal, while in \eqref{grad1} we exploited the properties of the eigenvalues $\lambda_k$, the equality $\mathcal{A}\bv^n= -\Delta\bv^n$ (that is  a well-known technique in the space-periodic case (cfr.  \cite{MNR}, \cite{DR})) and integration by parts.
\begin{lem}\label{lemma8}
Let $p\in \left(1,2\right)$, $\mu>0$ and $\bF\in L^\infty(0,T;L^{\frac{4}{p}}(\O))$. Let $\bv^n(t, \x)$ be the Galerkin approximation defined  in \eqref{def-galerkin}, then it holds
$$ \frac{1}{2}\|\bv^n(T)\|_{2,\O}^2 + \|\cD\bv^n\|_{p,\O_T}^p \leq CT\|\bF\|_{L^\infty(0,T;L^{\frac{4}{p}}(\O))}^{p'}+\frac{1}{2} \|\bv_0\|_{2,\O}^2 + \mu^{\frac{p}{2}}|\O_T| \ \mbox{ for all } \mu>0.$$
\end{lem}
\begin{proof} From \eqref{energy-d2} using the inequality
$$\ba{c}\displ \int_\O |\cD\bv^n|^p\,dx=\int_{|\cD\bv^n|^2\geq \mu} |\cD\bv^n|^p\,dx + \int_{|\cD\bv^n|^2< \mu} |\cD\bv^n|^p\,dx \VS\leq 2^{\frac{2-p}{2}} \int_\O (\mu+|\cD\bv^n|^2)^{\frac{p-2}{2}}|\cD\bv^n|^2\,dx+ \int_\O \mu^{\frac{p}{2}}\,dx
 = 2^{\frac{2-p}{2}} (\S_\mu(\cD\bv^n(t)), \cD\bv^n(t))+ \mu^{\frac{p}{2}}|\O|\ea$$
and integrating with respect to the time between $0$ and $T$, it follows
\be
\begin{split}
\frac{1}{2}\|\bv^n(T)\|_{2,\O}^2 + 2^{\frac{2-p}{2}}\|\cD\bv^n\|_{p,\O_T}^p &\leq \int_0^T\!\!\!\!\! \|\bF(t)\|_{\frac{4}{p}}\|\bv^n(t)\|_{\frac{4}{4-p}} dt
+\frac{1}{2} \|\bv_0\|_{2,\O}^2+  \mu^{\frac{p}{2}}|\O_T|\\
&\leq C\!\!\int_0^T \!\!\!\|\cD\bv^n(t)\|_{p}\,dt\, \|\bF\|_{L^\infty(0,T;L^{\frac{4}{p}}(\O))} +\frac{1}{2} \|\bv_0\|_{2,\O}^2+  \mu^{\frac{p}{2}}|\O_T|\\
& \leq CT^{\frac{1}{p'}}\|\cD\bv^n\|_{p,\O_T} \|\bF\|_{L^\infty(0,T;L^{\frac{4}{p}}(\O))} +\frac{1}{2} \|\bv_0\|_{2,\O}^2+  \mu^{\frac{p}{2}}|\O_T|\\
& \leq  \frac{1}{2}\|\cD\bv^n\|_{p,\O_T}^p + C T \|\bF\|_{L^\infty(0,T;L^{\frac{4}{p}}(\O))}^{p'}+\frac{1}{2} \|\bv_0\|_{2,\O}^2+ \mu^{\frac{p}{2}}|\O_T|
\end{split}\ee
where $C$ is a constant due to the Sobolev embedding $W^{1,p}\hookrightarrow L^{\frac{4}{4-p}}(\O)$, Korn's and Young's inequalities, that is unessential for our aims.
Finally, straightforward computations yield the assertion.\end{proof}

\begin{lem}\label{lemma1}
Let $p\in \left[\frac{5}{3},2\right)$. For every function $\bv$ sufficiently smooth and space-periodic with respect to $\Omega$ it holds
\be
\|\n\bv\|_{3,\O}^3 \leq C_{S}^3 \|\nabla \bv\|_{2,\O}^{\frac{9p-12}{3p-2}}\|\nabla\cD\bv\|_{\frac{4}{4-p},\O}^{\frac{6}{3p-2}}
\ee
 with $C_{S} $ independent of $\bv$.
\end{lem}
\begin{proof} Since
$$ \frac{1}{3}= \frac{3p-4}{3p-2}\frac{1}{2}+\frac{2}{3p-2}\frac{8-3p}{12},$$ 
interpolating we get
\be\label{conti-lem} 
\|\n\bv\|_{3,\O} \leq \|\nabla \bv\|_{2,\O}^{\frac{3p-4}{3p-2}}\|\nabla\bv\|_{\frac{12}{8-3p},\O}^{\frac{2}{3p-2}}.
\ee
Then the Korn inequality, the Sobolev embedding $W^{1, \frac{4}{4-p}}_{\per}(\O)\hookrightarrow L_{\per}^{\frac{12}{8-3p}}(\O)$ and the Poincar\'{e} inequality imply 
\begin{equation}
\|\n\bv\|_{3,\O} \leq C_{S} \|\nabla \bv\|_{2,\O}^{\frac{3p-4}{3p-2}}\|\nabla\cD\bv\|_{\frac{4}{4-p},\O}^{\frac{2}{3p-2}}
\end{equation}
where $C_{S}$ is a constant depending on $|\O|$ and $p$. Raising to the power $3$ the proof is concluded.
%
\end{proof}

We recall the following H\"older's  inequality (see e.g. \cite[ Theorem 2.12]{AF}).
\begin{ineq}\label{ineq}
Let $0<q<1$ and $q'=\frac{q}{q-1}$. If $f\in L^q(\O)$ and $0<\int_\O|g(x)|^{q'}\,dx<\infty$ then 
$$ \int_\O |f(x)g(x)|\geq \left(\int_\O|f(x)|^q\,dx\right)^{\frac{1}{q}}\left(\int_\O|g(x)|^{q'}\,dx\right)^{\frac{1}{q'}}.$$
\end{ineq}
As in \cite{CGM}, making use of Inequality \ref{ineq} we finely and quickly achieve an estimate for the gradient.
\begin{lem}\label{lemma6}
Let $p\in \left[\frac{5}{3},2\right)$ and $\bF\in L^\infty(0,T; L_{\per}^{\frac{4}{p}}(\O))$.  Let $\bv^n$ be the Galerkin approximation defined in \eqref{def-galerkin}, then for $t\in(0,T)$ 
it holds
 \be\ba{c}\displ
(\mu_1 + \|\nabla\bv^n(t)\|_2^{2})^{\frac{2-p}{2}}  \frac{d}{dt} \|\nabla\bv^n(t)\|_2^2+\frac{p-1}{2}\|\nabla\cD\bv^n(t)\|_{\frac{4}{4-p}}^2\VS
 \leq C_S^3  (\mu_1 + \|\nabla\bv^n(t)\|_2^{2})^{\frac{2-p}{2}}  \|\nabla \bv^n(t)\|_{2}^{\frac{9p-12}{3p-2}}\|\nabla\cD\bv^n(t)\|_{\frac{4}{4-p}}^{\frac{6}{3p-2}}
+\frac{1}{p-1}(\mu_1 + \|\nabla\bv^n(t)\|_2^{2})^{2-p}\|\bF(t)\|_{\frac{4}{p}}^2 
\ea\ee
 for all $\mu>0$, with  $C_S=C_S(p, |\O|)$ and $\mu_1:=\mu|\Omega|$.
\end{lem}
\begin{proof} Since
$$\label{gradS}\ba{l}\displ
\iO\! \!\n\S_\mu(\cD\bv^n)\!\cdot \!\n\cD\bv^n=\!\! \iO\! (\mu+|\cD\bv^n|^2)^{\frac{p-2}{2}}|\n\cD\bv^n|^2 
+ (p-2)\!\!\iO \!(\mu+|\cD\bv^n|^2)^{\frac{p-4}{2}}\!(\n\cD\bv^n \cdot\cD\bv^n)^2 \VS
 \geq
(p-1)\! \iO\!(\mu+ |\cD\bv^n|^2)^{\frac{p-2}{2}}|\n\cD\bv^n|^2,
\ea$$
and since (see e.g. \cite[pag. 225]{MNR} for more details)
$$|(\bv^n\!\cdot\!\n\bv^n\!, \Delta\bv^n)|\leq \|\n\bv^n\|_3^3 ,$$
 from \eqref{grad1}, we obtain
 \be\label{grad2}
\frac{1}{2} \frac{d}{dt} \|\nabla\bv^n(t)\|_2^2+(p-1)\iO(\mu+ |\cD\bv^n|^2)^{\frac{p-2}{2}}|\n\cD\bv^n|^2\,dx\leq \|\n\bv^n(t)\|_{3}^3+|(\bF\!,\Delta\bv^n)|.
\ee
Employing  Lemma \ref{lemma1} and the H\"older inequality on the right hand side and Inequality \ref{ineq} with $\frac{2}{p-2}$ and $\frac{2}{4-p}$ on the left-hand side, it follows
\be\label{lemma-grad}\ba{c}\displ
\frac{1}{2}  \frac{d}{dt} \|\nabla\bv^n(t)\|_2^2+(p-1)\frac{\|\nabla\cD\bv^n(t)\|_{\frac{4}{4-p}}^2}{(\mu_1 + \|\nabla\bv^n(t)\|_2^{2})^{\frac{2-p}{2}}} \leq C_S^3 \|\nabla \bv^n(t)\|_{2}^{\frac{9p-12}{3p-2}}\|\nabla\cD\bv^n(t)\|_{\frac{4}{4-p}}^{\frac{6}{3p-2}}
\VS\hfill +\|\bF(t)\|_{\frac{4}{p}} \|\nabla\cD\bv^n(t)\|_{\frac{4}{4-p}}
\ea\ee
where we set $\mu_1:=\mu|\Omega|$, then multiplying both sides by $(\mu_1 + \|\nabla\bv^n(t)\|_2^{2})^{\frac{2-p}{2}}$ and employing the Young inequality in the term $\|\bF(t)\|_{\frac{4}{p}} \|\nabla\cD\bv^n(t)\|_{\frac{4}{4-p}}$ the assertion follows and the proof is concluded.
\end{proof}
\section{Proof of Theorem \ref{thmA}}\label{sectionA} Let $\bv^n$ be the Galerkin sequence defined in \eqref{def-galerkin} fulfilling \eqref{galerkin} in $[0,T]$ and the initial condition $\bv^n(0)= P^n\bv_0$.
\paragraph{\bf Smallness of $\|\nabla\bv^n(t)\|_2$.}   Under the assumptions \rf{smallv0} we prove that 
\be\label{CIN}\dm\nabla v^n(t)\dm_2^2\leq \Lambda\,,\mbox{ for all }n\in\N \mbox{ and }t>0.\ee 
For sake of brevity let us introduce the following notations 
$$\varphi(t):=\| \nabla\bv^n(t)\|_2^2, \quad  D(t):=\|\nabla\cD\bv^n(t)\|_{\frac{4}{4-p}}^2.$$ By virtue of  Lemma \ref{lemma6}, recalling the definition of $\mu_1$, and \eqref{smallv0}$_2$ it holds 
\be\label{phi1} 
\frac{1}{2}(\mu_1 + \varphi(t))^{\frac{2-p}{2}}  \frac{d}{dt} \varphi(t)+\frac{p}{4}D(t)\leq 
C_{S}^3  (\mu_1 + \varphi(t))^{\frac{2-p}{2}}  \varphi(t)^{\frac{9p-12}{2(3p-2)}}D(t)^{\frac{3}{3p-2}}
+\frac{1}{p-1}(\mu_1 + \varphi(t))^{2-p}K^2 
\ee
for all $t\in[0,T]$. By means of Sobolev embedding and Poincar\'e inequality, for $p\in[\frac{5}{3},2)$ it holds  $\varphi(t)\leq C_S D(t) $  with $ C_S $ constant depending on $p$ and $\Omega$ and independent of $n$, then from \eqref{phi1} it follows 
\be\label{phi2} \ba{l}\displaystyle
\frac{1}{2}(\mu_1 + \varphi(t)\!)^{\frac{2-p}{2}}  \frac{d}{dt} \varphi(t)+\frac{p}{4}D(t)\!\leq\! 
C_{S}^3  (\mu_1 + \varphi(t)\!)^{\frac{2-p}{2}}  \varphi(t)^{\frac{9p-12}{2(3p-2)}-\frac{3p-5}{3p-2}}C_S^{\frac{3p-5}{3p-2}}D(t)^{\frac{3}{3p-2}+\frac{3p-5}{3p-2}}
 \VS\hfill
+\frac{1}{p-1}(\mu_1 + \varphi(t))^{2-p}K^2 \VS
\hspace{5cm}
=A(\mu_1 + \varphi(t))^{\frac{2-p}{2}}  \varphi(t)^{\frac{1}{2}}D(t)+\frac{1}{p-1}(\mu_1 + \varphi(t))^{2-p}K^2 
\ea\ee
where we set $A:=C_{S}^ {\frac{12p-11}{3p-2}}$. Hence,
\be\label{phi3}
\frac{1}{2}(\mu_1 + \varphi(t))^{\frac{2-p}{2}}  \frac{d}{dt} \varphi(t)+\left(\frac{p}{4} -A(\mu_1 + \varphi(t))^{\frac{2-p}{2}}  \varphi(t)^{\frac{1}{2}}\right) D(t)- \frac{1}{p-1}(\mu_1 + \varphi(t))^{2-p}K^2\leq 0.
\ee
Assuming $\dm \nabla v_0\dm_2^2<\Lambda$ estimate \rf{CIN} holds.    In this connection, by the continuity of $\dm \nabla v^n(t)\dm_2$, let $\overline t$ be the first instant such that $\dm \n v^n(\overline t)\dm_2^2=\Lambda$.  Since  \eqref{constant-Lambda} and we can also assume $\mu_1$ suitable small,
for $t=\overline t$  differential inequality  \rf{phi3} furnishes $$\frac{1}{2}(\mu_1 + \varphi(\overline t))^{\frac{2-p}{2}}  \frac{d}{dt} \varphi(\overline t)+ \frac{p}{8}  D(\overline t)- \frac{1}{p-1}(2\Lambda)^{2-p}K^2< 0\,.
$$
We can employ the Sobolev embedding $W_{\per}^{1,\frac{4}{4-p}}(\O)\hookrightarrow L_{\per}^{2}(\O)$ and the Korn inequality on the left hand side and we achieve
\be\label{p3}
\frac{1}{2}(\mu_1 + \varphi(\overline t))^{\frac{2-p}{2}}  \frac{d}{dt} \varphi(\overline t)+ \frac{p}{8}C_S^{-1}\varphi(\overline t)-\frac{1}{p-1}(2\Lambda)^{2-p}K^2< 0.
\ee
Since \eqref{constant-Lambda}
via \eqref{p3}   we conclude that 
 $$ \frac{d}{dt}\|\nabla\bv^n(\overline{t})\|_2<0\,,$$
which proves that $\overline t$ does not exist. Further, from  \eqref{phi3}, by straightforward computations via   \eqref{CIN}, it follows
\be\ba{c}\displ
\frac{1}{2} \frac{d}{dt} \|\nabla\bv^n(t)\|_2^2+ \frac p8 \|\nabla\cD\bv^n(t)\|_{\frac{4}{4-p}}^2 \leq \frac{1}{p-1}(\mu_1 + \Lambda)^{\frac{2-p}{2}}\|\bF(t)\|_{\frac{4}{p}}^2. 
\ea\ee
Then  integrating on $[0,T]$, for all $T$, we arrive at
\be\label{grad-dv} 
 \|\nabla\bv^n(T)\|_2^2+ \int_0^T \|\nabla\cD\bv^n(t)\|_{\frac{4}{4-p}}^2dt<+\infty 
\ee
uniformly with respect to $n\in\N$ and $\mu>0$.
\paragraph{\bf Other uniform estimates.} As a particular consequence of estimate \rf{CIN} we get   
  for all $T>0$ and for all measurable set $E\subseteq\Omega$ 
\be\label{S-mu-n}\intll0T\intl E |\S_\mu(\cD\bv^n)|^{p'}dxdt\leq \intll0T\intl E |\cD\bv^n|^{p}dxdt \leq T|E|^{\frac{2-p}{2}}\Lambda^\frac p2.
\ee
Let us investigate the regularity of $\bv^n_t$. From \eqref{vt}, by the  H\"older and the Young inequalities and the Sobolev embedding $W^{2, \frac{4}{4-p}}\hookrightarrow L^\infty$, it follows
\begin{equation}\begin{split}
\|\bv^n_t\|_2^2+ (\S_\mu(\cD\bv^n), \cD\bv^n_t) &\leq  \|\bv^n\|_{\infty}\|\n\bv^n\|_2 \|\bv^n_t\|_2+\|\bF\|_{\frac{4}{p}} \|\bv^n_t\|_2\\
&\leq  \|\bv^n\|_{\infty}^2\|\n\bv^n\|_2^2+\|\bF\|_{\frac{4}{p}}^2+\frac{1}{2}\|\bv^n_t\|_2^2\\
&\leq C_S\Lambda\|\nabla\cD\bv^n\|_{\frac{4}{4-p}}^2+\|\bF\|_{\frac{4}{p}}^2+\frac{1}{2}\|\bv^n_t\|_2^2.
\end{split}\end{equation}
Hence, we get
\begin{equation}
\frac{1}{2}\int_0^T\!\! \|\bv^n_t(t)\|_2^2dt+\frac{1}{p} \int_0^T\!\!\frac{d}{dt} \|(\mu+|\cD\bv^n(t)|^2)^{\frac{1}{2}}\|_p^p dt\leq C_S\Lambda\int_0^T \|\nabla\cD\bv^n(t)\|_{\frac{4}{4-p}}^2dt+\int_0^T\!\!\|\bF(t)\|_{\frac{4}{p}}^2dt,
\end{equation}
thus  by virtue of \eqref{grad-dv} we conclude
\begin{equation}\label{vt-final}\begin{split}
&  \|(\mu+|\cD\bv^n(T)|^2)^{\frac{1}{2}}\|_p^p+\frac{1}{2}\int_0^T\!\! \|\bv^n_t(t)\|_2^2dt \\
&\leq C_S\Lambda\!\!\int_0^T\!\! \|\nabla\cD\bv^n(t)\|_{\frac{4}{4-p}}^2dt+\int_0^T\!\!\|\bF(t)\|_{\frac{4}{p}}^2dt+\|(\mu+|\cD\bv^n(0)|^2)^{\frac{1}{2}}\|_p^pdt
<+\infty,
\end{split}\end{equation}
uniformly with respect to $n$ and $\mu$. 
 \paragraph{\bf Limit as $n\to +\infty$.} 
 From the uniform estimates in Lemma \ref{lemma8}, \eqref{CIN}, \eqref{grad-dv}, \eqref{S-mu-n},   \eqref{vt-final}, the compact embedding $W^{2,\frac{4}{4-p}}(\O)\hookrightarrow\hookrightarrow W^{1,2}(\O)$, that in particular holds in the considered interval for $p$, and the Aubin Lions Lemma, we get the existence of a subsequence, that we do not relabel, such that the following convergences hold
 \begin{subequations}\label{limit-n} \begin{align}
 &\bv^n \rightharpoonup \bv^\mu &&\mbox{weakly in } \ L^2(0,T;W^{2,\frac{4}{4-p}}(\O)),\\
 &\bv^n_t\rightharpoonup \bv_t^\mu  &&\mbox{weakly in } \ L^2(0,T;L^2(\O)),\\
 &\bv^n \to \bv^\mu &&\mbox{strongly in } \  L^2(0,T;W^{1,2}(\O)).\label{conv-strong}
 \end{align}
  \end{subequations}
 The lower semicontinuity of the $L^2$-norm with respect to weak convergence together with \eqref{CIN} imply that
\be\label{sup-n22} \dm\nabla v^\mu(t)\dm_2^2\leq \Lambda\,,\mbox{ for all }t>0.\ee
From the facts $\bv^\mu\in  L^2(0,T;W^{2,\frac{4}{4-p}}(\O))$, $\bv^\mu_t\in L^2(0,T;L^2(\O))$, we have that 
\be\label{CG} \bv^\mu \in C(0, T; J^{1,2}(\O)).\ee
Furthermore, as consequence of \eqref{conv-strong} it follows that for every $\mu> 0$
$$ \nabla \bv^n\to \nabla \bv^\mu \ \ \mbox{a.e. in } \O_T ,$$
then the continuity of the operator $\S_\mu$ implies that for all  $\mu> 0$
\be
\S_\mu(\cD\bv^n)\to \S_\mu(\cD\bv^\mu) \ \ \mbox{a.e. in } \O_T. 
\ee
Hence, by virtue of \rf{S-mu-n}, applying  Vitali's convergence theorem, we get
\be\label{conv-operatore}
\lim_{n\to +\infty} \int_{\O_T} \S_\mu(\cD\bv^n):\cD \bphi =\int_{\O_T} \S_\mu(\cD\bv^\mu): \cD\bphi \ \mbox{ for all } \bphi \in L^p(0, T; W^{1,p}(\O)).
\ee  
From \eqref{galerkin} we know that, fixed $n\geq m$, for $\omega\in X_m$  it holds
\be
(\bv_t^n, \bo) + (\bv^n\cdot\nabla \bv^n, \bo) +  (\S_\mu(\cD\bv^n), \cD\bo) =  (\bF, \bo),
\ee
then multiplying by $\eta \in C^1(0,T; \R)$ and integrating on the interval $(0,T)$ we have 
\be
\int_0^T  (\bv_t^n, \eta \bo) + (\bv^n\cdot\nabla \bv^n, \eta\bo) +  (\S_\mu(\cD\bv^n), \eta\cD\bo) \,dt=  \int_0^T (\bF, \eta\bo)\,dt.
\ee
Taking the limit as $n\to +\infty$, the convergences \eqref{limit-n} and \eqref{conv-operatore} imply
\begin{equation}\label{nuova}
\int_0^T  (\bv_t^\mu, \eta \bo) + (\bv^\mu\cdot\nabla \bv^\mu, \eta\bo) +  (\S_\mu(\cD\bv^\mu), \eta\cD\bo) \,dt=  \int_0^T (\bF, \eta\bo)\,dt
\end{equation}
for every $\eta \in C^1(0,T; \R)$ and $\omega \in \cup_{m\geq 1} X_m$. Since regularity property \rf{CG} in particular implies $\bv^\mu\in C([0,T);L^\frac{6p}{5p-6}(\Omega))$, by means of a density argument it follows
\be\label{form-final0}
  (\bv_t^\mu, \bphi) + (\bv^\mu\cdot\nabla \bv^\mu, \bphi) +  (\S_\mu(\cD\bv^\mu), \cD\bphi) \,dt=   (\bF, \bphi)\, 
\mbox{
 for all }\bphi\in   J^{1,p} (\O),\mbox{ a.e. in }t\in[0,T).\ee
 \paragraph{\bf Limit as $\mu \to 0$.}  
 For every $\mu>0$ we exhibited $\bv^\mu$ such that \eqref{form-final0} is fulfilled, in addition the  estimates  \eqref{CIN}, \eqref{grad-dv}, \eqref{S-mu-n},   \eqref{vt-final} are uniform with respect to $\mu$. Then employing the lower semicontinuity of the respective norms they still hold. Thus for  $\mu\to 0$ and for a suitable subsequence (that we do not relabel), we get 
  \begin{subequations}\label{limit-mu} \begin{align}
 &\bv^\mu \rightharpoonup \bv &&\mbox{weakly in } \ L^2(0,T;W^{2,\frac{4}{4-p}}(\O)),\\
 &\bv^\mu_t\rightharpoonup \bv_t  &&\mbox{weakly in } \ L^2(0,T;L^2(\O)),\\
 &\bv^\mu \to \bv &&\mbox{strongly in } \  L^2(0,T;W^{1,2}(\O)).\label{TV-c}
 \end{align}  \end{subequations} 
As a consequence of \rf{TV-c} we achieve 
$$ \n \bv^\mu\to \n\bv \mbox{ and } \bold S_\mu(\mathcal D v^\mu)\to \bold S(\mathcal D v)  \mbox{ a.e. in } \Omega_T.$$   
 Moreover, it is not difficult to realize that \rf{CIN} holds also in the case of $\bv$: 
\be\label{CIN-I}\dm\nabla \bv(t)\dm_2^2\leq \Lambda\,,\mbox{ for all }t>0. \ee
Since \rf{S-mu-n} holds for $\bv^\mu$ too, applying  Vitali's convergence theorem,  we get the limit property 
\be\label{conv-operatore2}
\lim_{\mu\to 0} \int_{\O_T} \S_\mu(\cD\bv^\mu):\cD \bphi =\int_{\O_T} \S(\cD\bv): \cD\bphi \ \mbox{ for all } \bphi \in L^p(0, T; W^{1,p}(\O)).
\ee  
 Again the statements $\bv\in  L^2(0,T;W^{2,\frac{4}{4-p}}(\O))$, $\bv_t\in L^2(0,T;L^2(\O))$ imply
$$ \bv \in C(0, T; J^{1,2}(\O)).$$ 
Finally, taking the limit as $\mu\to 0$ in \eqref{nuova}, employing the convergences \eqref{limit-mu} and \eqref{conv-operatore2}, and by the same arguments used in the passage to the limit as $n\to+\infty$ we establish
\be\label{form-final01}
 (\bv_t, \bphi) + (\bv\cdot\nabla \bv, \bphi) +  (\S(\cD\bv), \cD\bphi) \,dt=   (\bF, \bphi)\, 
\mbox{
 for all }\bphi\in   J^{1,p} (\O),\mbox{ a.e. in }t\in[0,T).\ee 
 \chiu
 
\section{Proof of Theorem \ref{thm1}}\label{section1}
Let $\b_0$ be in $J_{\per}^{1,2}(\O)$  such that $\|\nabla\b_0\|_2^2\leq \Lambda$ with $\Lambda$ the constant determined in \eqref{constant-Lambda}. By virtue of Theorem \ref{thmA} there exists a solution, denoted by $\b(\bx,t)$, corresponding to $\b_0$ and $\bF$ to system \eqref{system-mu} for any $\mu>0$.  For $m,k\in \N$ with $m>k$, set $\b'(t, \bx):=\b(t+(m-k)\mathcal T, \bx)$ with $t\geq 0$. The difference $\bu(t, \bx) := \b'(t,\bx)-\b(t,\bx)$ is solution a.e. in $(0,  T)\times\O$ to the system
\be\label{system-u}\ba{l}\displ
\bu_t-\div(\S_\mu(\b')-\S_\mu(\b))=-\bu\cdot\n\b'-\b\cdot\n\bu-\n(\pi_{\b'}-\pi_{\b}),\VS
\div \bu=0
\ea\ee
where we exploited the periodicity of $\bF$.
We multiply by $\bu$ equation \eqref{system-u}$_1$, integrating over $\O$ it holds
$$\frac12\frac{d}{dt} \|\bu\|_2^2 + \iO (\S_\mu(\b)-\S_\mu(\b'))\cdot \cD(\b-\b')\,dx\leq |(\bu\cdot\n\bu,\b')|\,.$$
By virtue of inequality \rf{property-diff} and H\"older's inequality, we get
$$\frac12\frac{d}{dt} \|\bu\|_2^2 + 3^{\frac{p-2}2}\iO|\cD\bu|^2(\mu+ |\cD\b|^2+|\cD\b'|^2)^\frac{p-2}2\,dx\leq   \|\bu\|_{\frac{12}{8-3p}}\|\n\bu\|_{\frac{4}{4-p}}\|\b'\|_{\frac{6}{3p-4}}\,.$$
Employing   the Sobolev embeddings $W^{1, \frac{4}{4-p}}\hookrightarrow L^{\frac{12}{8-3p}}$ and $W^{1,2}\hookrightarrow L^{\frac{6}{3p-4}}$ and the Poincar\'e inequality and \rf{gradient2} on the right-hand side, we deduce 
\be\label{holder-d}\ba{l}\displ
\frac12\frac{d}{dt} \|\bu\|_2^2 +3^{\frac{p-2}2} \iO|\cD\bu|^2(\mu+ |\cD\b|^2+|\cD\b'|^2)^{\frac{p-2}{2}}\,dx\leq |(\bu\cdot\n\bu,\b')| \leq C_S\Lambda^\frac12  \|\nabla \bu\|_{\frac{4}{4-p}}^2,
\ea\ee
moreover using Inequality \ref{ineq} on the left-hand side with $\frac{2}{p-2}$ and $\frac{2}{4-p}$ and the Korn inequality, we obtain
$$
\frac{d}{dt} \|\bu\|_2^2 +\frac{3^{\frac{p-2}2}C_K\|\n\bu\|_{\frac{4}{4-p}}^2}{ (\mu_1 + \|\nabla\b(t)\|_2^2+ \|\nabla\b'(t)\|_2^2)^\frac{2-p}2} \leq  C_S \Lambda^\frac12 \|\n\bu\|_{\frac{4}{4-p}}^{2}
$$

where $C_K$ is the constant of the Korn inequality, hence 

$$ \frac12\frac{d}{dt} \|\bu\|_2^2 +\Big[\frac{3^{\frac{p-2}2}C_K}{ (\mu_1 + \|\nabla\b(t)\|_2^2+ \|\nabla\b'(t)\|_2^2)^\frac{2-p}2} - C_S\Lambda^\frac12\Big]\|\n\bu\|_{\frac{4}{4-p}}^2 \leq 0.$$
Then,  by virtue of \eqref{constant-Lambda}
and employing  the Sobolev embedding $W^{1, \frac{4}{4-p}}(\O)\hookrightarrow L^2(\O)$, it results 
\be
\frac{d}{dt} \|\bu\|_2^2 +\frac{3^{\frac{p-2}2}C_K-C_S\Lambda^\frac12(\mu_1+2\Lambda)^{\frac{2-p}{2}}}{(\mu_1 +2\Lambda)^{\frac{2-p}{2}}} \|\bu\|_2^2\leq 0.
\ee
Integrating the latter between $0$ and $k\mathcal T$, it follows

\be
\|\bu(k\mathcal T)\|_2 \leq e^{-\frac{3^{\frac{p-2}2}C_K-C_S\Lambda^\frac12(\mu_1+2\Lambda)^{\frac{2-p}{2}}}{2(\mu_1 +2\Lambda)^{\frac{2-p}{2}}}k\mathcal T} \|\bu(0)\|_2,
\ee
and by recalling the definition of $\bu$ yields
\be\label{cauchy}\ba{l}\displ
\|\b(m\mathcal T)-\b(k\mathcal T)\|_2\leq  e^{-\frac{3^{\frac{p-2}2}C_K-C_S\Lambda^\frac12(\mu_1+2\Lambda)^{\frac{2-p}{2}}}{2(\mu_1 +2\Lambda)^{\frac{2-p}{2}}}k\mathcal T} \|\b((m-k)\mathcal T)-\b_0\|_2\VS
\leq 2 e^{-\frac{3^{\frac{p-2}2}C_K-C_S\Lambda^\frac12(2\mu_1+2\Lambda)^{\frac{2-p}{2}}}{2(\mu_1 +2\Lambda)^{\frac{2-p}{2}}}k\mathcal T} C_S\Lambda^\frac{1}{2},
\ea\ee
where we used the Poincar\'e inequality and Theorem \ref{thmA}.
Set $\b_n(x):=\b(n\mathcal T, x)$.
From \eqref{cauchy} it results that  $\b_n(\x)$ is a Cauchy sequence in $J(\O)$ and thus there exists $\bv_0\in J(\O)$ such that  $\b_n\to\bv_0$ strongly in $L^2(\O)$. In addition by virtue of Theorem \ref{thmA} $\b_n(\x)$ fulfills 
$$\|\nabla \b_n\|_2^2\leq \Lambda$$
thus it holds 
\begin{equation}
\nabla \b_n \rightharpoonup \nabla \bv_0 \ \ \mbox{weakly in } \ L^2(\O),
\end{equation} 
and by means of the lower semicontinuity of the $L^2$-norm we achieve
\begin{equation}\label{v0mu}
\|\nabla \bv_0\|_2^2\leq \Lambda.
\end{equation}
Function $\bv_0$ enjoys the assumptions of Theorem \ref{thmA}, hence there exists $\bv(t, \bx)$ solution corresponding to $\bv_0(\bx)$ and $\bF(t, \bx)$ to system \eqref{system-mu} for any $\mu>0$. The solution $\bv(t, \bx)$ is time periodic of period $\mathcal T$. To prove the periodicity of $\bv(t, \bx)$, set $\bv'(t, \bx):= \b(t+n\mathcal T, \bx)$ and consider $\bw(t, \bx):=\bv'(t, \bx) -\bv(t, \bx)$, which is solution of the following system
\be\ba{l}\displ
\bw_t-\div(\S_\mu(\bv')-\S_\mu(\bv))=-\bv'\cdot\n\bv'+\bv\cdot\nabla\bv-\n\pi_{\bw},\VS
\div \bw=0,\VS
\bw(\bx, 0)=\b(n\mathcal T, \bx)-\bv_0(\bx).
\ea\ee
Multiply by $\bw$, integrate over $\O$ and by analogous arguments made for  $\bu(t,\bx)$, concerning the system \eqref{system-u}, we deduce 
\be\label{vudoppio}
\|\bw(t)\|_2 \leq  
e^{-\frac{3^{\frac{p-2}{2}}C_K-C_S\Lambda^{\frac{1}{2}}(\mu_1+2\Lambda)^{\frac{2-p}{2}}}
{2(\mu_1 +2\Lambda)^{\frac{2-p}{2}}}t} 
 \|\bw(0)\|_2,
\ee
hence
\be\ba{l}\displ
\|\b(t+n\mathcal T)-\bv(t)\|_2 \leq  e^{-\frac{3^{\frac{p-2}2}C_K-C_S\Lambda^{\frac{1}{2}}(\mu_1+2\Lambda)^{\frac{2-p}{2}}}{2(\mu_1 +2\Lambda)^{\frac{2-p}{2}}}t}  \|\b(n\mathcal T)-\bv_0\|_2.
\ea\ee
Set $t=\mathcal T$. Taking the limit as $n$ goes to infinity and using the strong convergence of $\b_n$ to $\bv_0$ in $L^2(\O)$, it follows
\be
\|\bv_0-\bv(\mathcal T)\|_2=0.
\ee
Thus for any $\mu>0$ we exhibited the existence of an initial data $\bv_0^\mu$ and of a corresponding regular time-periodic solution $\bv^\mu$  to \eqref{system-mu} with period $\mathcal T$. 
 Following the arguments in the proof of Theorem \ref{thm1} we derive uniform estimates that allow the passage to the limit as $\mu\to 0$. The limit $\bv$ is a regular solution to \eqref{system} in the sense of Definition \ref{def-d2}. It remains to be proved that $\bv$ is time-periodic.
In particular $\bv^\mu$ enjoys the following uniform estimate  with respect to $\mu$ (see \eqref{sup-n22})
$$\|\nabla\bv^\mu(\mathcal T)\|^2_{L^{2}(\O)}\leq \Lambda.$$ 
Thus from the compact embedding $W^{1,2}(\O)\hookrightarrow\hookrightarrow L^2$, the limit $\bv$ fulfills
\be\label{T}
\bv^\mu(\mathcal T)\to\bv(\mathcal T) \ \ \mbox{strongly in } L^2(\O).
\ee
Moreover, by virtue of \eqref{v0mu} and compact embedding there exists $\bv_0$ such that the sequence $\bv_0^\mu$ enjoys in particular
\be\label{zero}
\bv^\mu_0\to\bv_0 \ \ \mbox{strongly in } L^2(\O).
\ee
Therefore the time-periodicity of $\bv^\mu$ and the convergences \eqref{T} and \eqref{zero} imply that $\bv$ is time periodic with period $\mathcal T$. 
For the uniqueness, consider $\bv$ and $\bv'$ two time-periodic solutions with period $\mathcal T$ corresponding to $\bF$, then define $\bw(t, \bx):=\bv'(t, \bx) -\bv(t, \bx)$ and following the arguments employed in \eqref{vudoppio} (where the same symbols have different meaning) we get
$$\lim_{n\to +\infty} \|\bv(n\mathcal T)-\bv'(n\mathcal T)\|_2=0,$$
but this contradicts  the periodicity of the two solutions.
\chiu

\section{Proof of Theorem \ref{extinction}}
Theorem \ref{thm1} ensures the existence of time-periodic solutions $\bv$ to system \eqref{system}.
Taking $\bv$ as test function in the weak formulation of \eqref{system}, then employing Sobolev embedding, the Korn, the Young inequalities and since $\bF(t)\in L^2(\Omega)$ due to the assumption $\bF(t) \in L^{\frac{4}{p}}(\Omega)$ a.e. in $(0,T)$ with $p<2$, we get
\be
\frac{1}{2}\frac{d}{dt}\|\bv(t)\|_2^2+\|\cD\bv(t)\|_p^p\leq \|\bF(t)\|_2\|\bv(t)\|_2\leq C\|\bF(t)\|_2^{p'}+\frac{1}{2}\|\cD\bv(t)\|_p^p,
\ee
i.e.
\be\label{ext}
\frac{d}{dt}\|\bv(t)\|_2^2+\|\cD\bv(t)\|_p^p\leq C\|\bF(t)\|_2^{p'}.
\ee
Since $\|\bF(t)\|_2=0$ in $[t_f, \mathcal T]$, considering  \eqref{ext} between $t_f$ and $t$ after employed Sobolev embedding,   it follows
\be
\frac{d}{dt}\|\bv(t)\|_2^2+C_S^{-1}\|\bv(t)\|_2^p\leq 0,
\ee 
then integrating in $[t_f, t]$
\be
\|\bv(t)\|_2^{2-p}\leq -(2-p)C_S^{-1}(t-t_f)+\|\bv_0\|_2^{2-p}.
\ee
Thus it holds
\be
\|\bv(t)\|_2=0 \ \ \mbox{for all} \ t\!:  \ t_f+\frac{\|\bv_0\|_2^{2-p}}{C_S^{-1}(2-p)}\leq t\leq \mathcal T.
\ee
\chiu
\section{Proof of Corollary \ref{cor}}
By virtue of Theorem \ref{thm1} since $\bF$ is independent of the time, for every $\mathcal T>0$, there exists $\bv(t, x)$ as  unique time-periodic solution to system \eqref{system} with period $\mathcal T$.  
Let $\bv^{ \mu}$ the  approximation constructed as in the proof of Theorem \ref{thm1} that converges to $\bv(t, x)$. In what follows we 
neglect the indexes $\mu$. We set $\bu:=\bv^\mu(t,x)-\bv^\mu(s,x)$. We consider the equation for $\bu$ deduced from \rf{system-mu}. Hence we consider
$$ \ba{l}\displ
\bu_t-\div(\S_\mu(\bv(t))-\S_\mu(\bv(s)))=-\bu\cdot\n\bv(t)-\bv(s)\cdot\n\bu-\n\pi_{\bu},\VS
\div \bu=0
\ea$$
where we employ the independence of $t$ of  data $\bF$.

 By virtue of the regularity of $\bv^\mu$, multiplying by $\bu (t,x)$ and integrating by parts,     almost for all $t>0$ we get
$$
\frac{1}{2}\frac{d}{dt}\|\bu(t)\|_2^2+ 3^{\frac{p-2}2}(p-1)\int_\Omega (\mu+|\cD\bv(t)|^2+|\cD\bv(s)|^2)^{\frac{p-2}{2}}|\cD\bu|^2\leq |(\bu\cdot\nabla\bv, \bu)|.
$$
By  H\"older's inequality, the Sobolev embedding $W^{1, \frac{4}{4-p}}_{\per}(\Omega)\hookrightarrow L_{\per}^4(\Omega)$, the Korn inequality and \eqref{gradient2} on the right-hand side and Inequality \ref{ineq} on the left-hand side, we obtain
\begin{equation}
\frac{1}{2}\frac{d}{dt}\|\bu\|_2^2+3^{\frac{p-2}2}(p-1)\frac{\|\cD\bu\|_{\frac{4}{4-p}}^2}{(\mu_1+\|\cD\bv(t)\|_2^2+\|\cD\bv(s)\|^2)^{\frac{2-p}{2}}}\leq\|\nabla\bv\|_2 \|\bu\|_4^2\leq C_SC_K\Lambda^{\frac{1}{2}} \|\cD\bu\|_{\frac{4}{4-p}}^2 .
\end{equation}
Recalling that  $\mu_1:=\mu|\Omega|$, then 
\eqref{gradient2} and straightforward computations  imply
 \begin{equation}
\frac{1}{2}\frac{d}{dt}\|\bu\|_2^2+3^{\frac{p-2}2}(p-1)\frac{\|\cD\bu\|_{\frac{4}{4-p}}^2}{(\mu_1+\Lambda)^{\frac{2-p}{2}}}\leq C_SC_K\Lambda^{\frac{1}{2}} \|\cD\bu\|_{\frac{4}{4-p}}^2.
\end{equation}
Therefore, by virtue of  \eqref{constant-Lambda} and since we can fix $\mu$ arbitrarily small, we have
  $3^\frac{p-2}2(p-1)-(\mu_1+\Lambda)^{\frac{2-p}{2}}C_SC_K\Lambda^{\frac{1}{2}}>0$, then the Sobolev embedding $W^{1, \frac{4}{4-p}}_{\per}(\Omega)\hookrightarrow L_{\per}^2(\Omega)$ gives
 \begin{equation}
\frac{1}{2}\frac{d}{dt}\|\bu\|_2^2+\frac{3^{\frac{p-2}2}(p-1)-C_SC_K(\mu_1+\Lambda)^{\frac{2-p}{2}}\Lambda^{\frac{1}{2}}}{(\mu_1+\Lambda)^{\frac{2-p}{2}}}\|\bu\|_{2}^2\leq 0.
\end{equation}
Integrating between $h$ and $h+ \mathcal T$,  we get
$$
\|\bu(h+\mathcal T)\|_2 ^2=\|\bv(h+\mathcal T)-\bv(s)\|_2^2=\|\bv(h )-\bv(s)\|_2^2\leq \exp[ -\mathcal T]\|\bv(h)-\bv(s)\|_2^2.
$$ 
which proves that $\bv^\mu(h)=\bv^\mu(s)$. Since it is true for all $h\in[0,\mathcal T] $, then the solution $\bv^\mu$ is a steady solution. The same is true   for the limit $\bv$.\chiu

\nocite{*}
\bibliographystyle{amsplain}
\bibliography{periodici-reference}

\end{document}